\documentclass[12pt,a4paper]{amsart}

\newtheorem{theorem}{Theorem}
\newtheorem{corollary}[theorem]{Corollary}

\newtheorem{proposition}[theorem]{Proposition}

\usepackage{amsfonts}
\usepackage[figuresright]{rotating}
\usepackage{amssymb}
\usepackage{amsmath}
\usepackage{amsmath}
\usepackage{graphics}
\title{Permutations of type $B$ with fixed number of descents and minus signs}
\author{Katarzyna Kril}
\author{Wojciech M{\l}otkowski}
\thanks{
W.~M. is supported by the Polish
National Science Center grant No. 2016/21/B/ST1/00628.}
\address{Instytut Matematyczny,
Uniwersytet Wroc{\l}awski,
Plac~Grunwaldzki~2/4,
50-384 Wroc{\l}aw, Poland}
\email{katarzyna.kril03@gmail.com}
\email{mlotkow@math.uni.wroc.pl}
\subjclass[2010]{Primary 05A05; Secondary  20B35}
\keywords{Descents in permutations, Eulerian numbers, permutations of type B}

\begin{document}

\begin{abstract}
We study three dimensional array of numbers $B(n,k,j)$,
$0\le j,k\le n$, where $B(n,k,j)$ is the number of type $B$
permutations of order $n$ with $k$ descents and $j$ minus signs.
We prove in particular, that $b(n,k,j):=B(n,k,j)/\binom{n}{j}$
is an integer and provide two combinatorial interpretations for these numbers.
\end{abstract}

\maketitle

\section*{Introduction}

Let $B(n,k,j)$ denote the number of type $B$ permutations $(0,\sigma_1,\ldots,\sigma_n)$
which have $k$ descents and $j$ minus signs.
We study properties of the three-dimensional array $B(n,k,j)$, $0\le j,k\le n$.
Some of these properties appear in the work of Brenti \cite{brenti}.
In particular he computed the three-variable generating function
and proved real rootedness of some linear combinations
of the polynomials $P_{n,j}(x):=\sum_{k=0}^{n}B(n,k,j)x^k$
(Corollary~3.7 in \cite{brenti}, see also Corollary~6.9 in \cite{branden}).
Here we will prove that the numbers $b(n,k,j):=B(n,k,j)/\binom{n}{j}$ are also integers.
We provide two combinatorial interpretations of them.

For a subset $U\subseteq\{1,\ldots,n\}$ and $0\le k\le n$
let $\mathcal{B}_{n,k,U}$ denote the family of all
type $B$ permutations $\sigma=(0,\sigma_1,\ldots,\sigma_n)$ that $\sigma$ has $k$ descents
and satisfy: $\sigma_i<0$ iff $|\sigma_i|\in U$.
We will show (Theorem~\ref{theorembnku}) that the cardinality of $\mathcal{B}_{n,k,U}$
is $b(n,k,|U|)$.

Conger \cite{conger2010,congerthesis} defined the refined Eulerian number
$\left\langle n\atop k\right\rangle_{j}$ as the cardinality of the set
$\mathcal{A}_{n,k,j}$
of all type $A$ permutations $\tau=(\tau_1,\ldots,\tau_{n})$ such that $\tau_1=j$ and
$\tau$ has $k$ descents.
He proved many interesting properties of these numbers,
like direct formula, asymptotic behavior, lexicographic unimodality, formula for the generating function
and real rootedness of the corresponding polynomials.
It turns out that for $0\le j,k\le n$ we have $b(n,k,j)=\left\langle n+1\atop k\right\rangle_{j+1}$.
We will prove this equality
providing a bijection $\mathcal{A}_{n+1,k,j+1}\to\mathcal{B}_{n,k,U}$,
where $U=\{1,\ldots,j\}$ (Theorem~\ref{theorembijection}).
The array $b(n,k,1)$, $1\le k\le n$, appears in OEIS \cite{oeis} as A120434.
It also counts permutations $\sigma\in\mathcal{A}_{n}$
which have $k-1$ big descents, i.e. such descents $\sigma_i>\sigma_{i+1}$ that $\sigma_{i}-\sigma_{i+1}\ge2$.

Conger proved that the polynomials $p_{n,j}(x):=\sum_{k=0}^{n} b(n,k,j)x^k$
have only real roots  (Theorem~5 in \cite{conger2010}). Br\"{a}nd\'{e}n \cite{branden2}
showed something stronger: for every $n\ge1$
the sequence of polynomials $\{p_{n,j}(x)\}_{j=0}^{n}$
is interlacing, in particular for every $c_0,c_1,\ldots,c_n\ge0$
the polynomial $c_{0}p_{n,0}(x)+c_{1}p_{n,1}(x)+\ldots+c_{n}p_{n,n}(x)$
has only real roots.
Here we remark, that $P_{n,j}(x)=\binom{n}{j}p_{n,j}(x)$,
so the polynomials $P_{n,j}(x)$ admit the same property,
which is a generalization of Corollary~3.7 in \cite{brenti}
and of Corollary~6.9 in \cite{branden}.

\section{Preliminaries}

For a sequence $(a_0,\ldots,a_s)$, $a_i\in\mathbb{R}$, the \textit{number of descents},
denoted $\mathrm{des}(a_0,\ldots,a_s)$,
is defined as the cardinality of the set
$\big\{i\in\{1,\ldots,s\}:a_{i-1}>a_{i}\big\}.$
We will use the \textit{Iverson bracket:} $[p]:=1$ if the statement $p$
is true and $[p]:=0$ otherwise, see~\cite{gkp}.

Denote by $\mathcal{A}_{n}$ the group of permutations of the set $\{1,\ldots,n\}$.
We will identify $\sigma\in\mathcal{A}_{n}$ with the sequence $(\sigma_1,\ldots,\sigma_n)$
(we will usually write $\sigma_k$ instead of $\sigma(k)$).
For $0\le k\le n$ we define $\mathcal{A}_{n,k}$ as the set of those $\sigma\in\mathcal{A}_{n}$
such that the sequence $(\sigma_1\ldots,\sigma_{n})$ has $k$ descents.
Then the classical \textit{type $A$ Eulerian number} $A(n,k)$ (see entry A123125 in OEIS) is defined as the cardinality of $\mathcal{A}_{n,k}$.
We have the following recurrence relation:
\begin{equation}\label{arecurrence}
A(n, k)=(n-k)A(n-1,k-1)+(k+1)A(n-1,k)
\end{equation}
for $0<k<n$, with the boundary conditions:
$A(n,0)=1$ for $n\geq0$ and $A(n,n)=0$ for $n\geq1$.
These numbers can be expressed as:
\begin{equation}\label{aformula}
A(n,k)=\sum_{i=0}^{k}(-1)^{k-i}\binom{n+1}{k-i}(i+1)^{n}.
\end{equation}

For the Eulerian polynomials
\[
P^{\mathrm{A}}_{n}(t):=\sum_{k=0}^{n}A(n, k)t^k
\]
the exponential generating function is equal to
\begin{equation}\label{agenfunction}
f^{\mathrm{A}}(t,z):=
\sum_{n=0}^{\infty}\frac{P^{\mathrm{A}}_{n}(t)}{n!}z^n
=\frac{(1-t)e^{(1-t)z}}{1-t e^{(1-t)z}}.
\end{equation}

By $\mathcal{B}_{n}$ we will denote the group of such permutations $\sigma$ of the set
\[
\{-n,\ldots,-1,0,1,\ldots,n\}
\]
such that $\sigma$ is odd, i.e. $\sigma(-k)=-\sigma(k)$
for every $-n\le k\le n$. Then $|\mathcal{B}_{n}|=2^{n}n!$.
We will identify $\sigma\in\mathcal{B}_{n}$
with the sequence $(0,\sigma_1,\ldots,\sigma_n)$.
For $\sigma\in\mathcal{B}_{n}$ we define $\mathrm{des(\sigma)}$ (resp. $\mathrm{neg}(\sigma)$) as the number
of descents (resp. of negative numbers) in the sequence $(0,\sigma_1,\ldots,\sigma_n)$.
For $0\le k,j\le n$ we define sets
\begin{align*}
\mathcal{B}_{n,k}&:=\{\sigma\in\mathcal{B}_{n}:\mathrm{des}(\sigma)=k\},\\
\mathcal{B}_{n,k,j}&:=\{\sigma\in\mathcal{B}_{n}:\mathrm{des}(\sigma)=k,\mathrm{neg}(\sigma)=j\},
\end{align*}
and the numbers $B(n,k):=|\mathcal{B}_{n,k}|$ (\textit{type $B$ Eulerian numbers}, see entry A060187 in OEIS),
$B(n,k,j):=|\mathcal{B}_{n,k,j}|$.
The numbers $B(n,k)$
satisfy the following recurrence relation:
\begin{equation}\label{brecurrence}
B(n,k)=(2n-2k+1)B(n-1,k-1)+(2k+1)B(n-1,k),
\end{equation}
$0<k<n$, with the boundary conditions $B(n,0)=B(n,n)=1$,
and can be expressed as
\begin{equation}
B(n,k)=\sum_{i=0}^{k}(-1)^{k-i}\binom{n+1}{k-i}(2i+1)^{n}.
\end{equation}

The type $B$ Eulerian polynomials are defined by
\[
P^{\mathrm{B}}_{n}(t):=\sum_{k=0}^{n}B(n, k)t^k,
\]
and the corresponding exponential generating function is equal to
\begin{equation}\label{bbgeneratingfunction}
f^{\mathrm{B}}(t,z)
:=\sum_{n=0}^{\infty}\frac{P^{\mathrm{B}}_n(t)}{n!}z^n
=\frac{(1-t)e^{(1-t)z}}{1-t e^{2(1-t)z}}.
\end{equation}

\section{Descents and signs in type $B$ permutations}\label{sectioncapitalb}

This section is devoted to the numbers $B(n,k,j):=|\mathcal{B}_{n,k,j}|$.
First we observe the following symmetry.

\begin{proposition}
For $0\le j,k\le n$ we have
\begin{equation}\label{bbnkjsymmetry}
B(n,k,j)=B(n,n-k,n-j).
\end{equation}
\end{proposition}

\begin{proof}
It is sufficient to note that the map
\[
(0,\sigma_1,\ldots,\sigma_n)\mapsto(0,-\sigma_1,\ldots,-\sigma_n)
\]
is a bijection of $\mathcal{B}_{n,k,j}$ onto $\mathcal{B}_{n,n-k,n-j}$.
\end{proof}

Now we provide two summation formulas.

\begin{proposition}
\begin{align}
\sum_{j=0}^{n}B(n,k,j)&=B(n,k),\label{bbnkjsum1}\\
\sum_{k=0}^{n}B(n,k,j)&=\binom{n}{j}n!.\label{bbnkjsum2}
\end{align}
\end{proposition}

\begin{proof}
The former sum counts all $\sigma\in\mathcal{B}_{n}$ which have $k$ descents,
while the latter counts all $\sigma\in\mathcal{B}_{n}$ which have
$j$ minus signs in the sequence $(\sigma_1,\ldots,\sigma_n)$.
\end{proof}

From Corollary~4.4 in \cite{borowiec2016} we have also
\begin{align}
\sum_{\substack{j=0\\ j\,\,\mathrm{even}}}^{n}B(n,k,j)
&=\frac{1}{2}B(n,k)
+\frac{(-1)^{k}}{2}\binom{n}{k},\\
\sum_{\substack{j=0\\ j\,\,\mathrm{odd}}}^{n}B(n,k,j)
&=\frac{1}{2}B(n,k)
-\frac{(-1)^{k}}{2}\binom{n}{k},
\end{align}
see A262226 and A262227 in OEIS.

Now we present the basic recurrence relations for the numbers $B(n,k,j)$.

\begin{theorem}\label{theorembrecurrence}
The numbers $B(n,k,j)$ admit the following recurrence:
\begin{equation}\label{bbnkjrecurrence}
\begin{gathered}
B(n,k,j)=(k+1)B(n-1,k,j)+(n-k)B(n-1,k-1,j)\\
+k B(n-1,k,j-1)+(n-k+1)B(n-1,k-1,j-1)
\end{gathered}
\end{equation}
for $0< k,j< n$, with boundary conditions:
\begin{align}
B(n,0,j)&=[j=0],&B(n,n,j)&=[j=n],\label{bbboundary1}\\
B(n,k,0)&=A(n,k),&
B(n,k,n)&=A(n,n-k)\label{bbboundary2}
\end{align}
for $0\le k,j\le n$.
\end{theorem}

Equality (\ref{bbnkjrecurrence}) remains true for $0\le j,k\le n$
under convention that $B(n,k,j)=0$ whenever $j\in\{-1,n+1\}$ or $k\in\{-1,n+1\}$.

\begin{proof}
For $(\sigma_0,\ldots,\sigma_n)\in\mathcal{B}_{n}$, $n\ge1$, we define
\[
\Lambda\sigma:=(\sigma_0,\ldots,\widehat{\sigma}_i,\ldots,\sigma_{n})\in\mathcal{B}_{n-1},
\]
where $i$ is such that $\sigma_i=\pm n$, and the symbol ``$\widehat{\sigma}_i$" means, that
the element ${\sigma}_i$ has been removed from the sequence.

For given $\sigma\in\mathcal{B}_{n,k,j}$, $0<k,j<n$, we have four possibilities:
\begin{itemize}
\item
$\sigma_i=n$ and either $i=n$ or $\sigma_{i-1}>\sigma_{i+1}$, $1\le i<n$.
Then $\Lambda\sigma\in\mathcal{B}_{n-1,k,j}$.

\item $\sigma_i=n$ and $\sigma_{i-1}<\sigma_{i+1}$, $1\le i<n$.
Then $\Lambda\sigma\in\mathcal{B}_{n-1,k-1,j}$.

\item $\sigma_i=-n$ and $\sigma_{i-1}>\sigma_{i+1}$, $1\le i<n$.
Then $\Lambda\sigma\in\mathcal{B}_{n-1,k,j-1}$.

\item $\sigma_i=-n$ and either $i=n$ or $\sigma_{i-1}<\sigma_{i+1}$, $1\le i<n$.
Then $\Lambda\sigma\in\mathcal{B}_{n-1,k-1,j-1}$.
\end{itemize}

Now, suppose we are given a fixed $\tau=(\tau_0,\ldots,\tau_{n-1})$
which belongs to one of the sets $\mathcal{B}_{n-1,k,j}$, $\mathcal{B}_{n-1,k-1,j}$,
$\mathcal{B}_{n-1,k,j-1}$ or $\mathcal{B}_{n-1,k-1,j-1}$.
We are going to count all $\sigma\in\mathcal{B}_{n,k,j}$ such that $\Lambda\sigma=\tau$.

If $\tau\in\mathcal{B}_{n-1,k,j}$ then
we should either put $n$ at the end of $\tau$, or insert into a descent of $\tau$,
i.e. between $\tau_{i-1}$ and $\tau_{i}$, where $1\le i\le n-1$, $\tau_{i-1}>\tau_{i}$,
therefore we have $k+1$ possibilities.

Similarly, if $\tau\in\mathcal{B}_{n-1,k-1,j}$
then we construct $\sigma$ by inserting $n$ between $\tau_{i-1}$ and $\tau_{i}$, $1\le i\le n-1$,
where $\tau_{i-1}<\tau_{i}$. For this we have $n-k$ possibilities.

Now assume that $\tau\in\mathcal{B}_{n-1,k,j-1}$.
Then we should insert $-n$ between
$\tau_{i-1}$ and $\tau_{i}$, $1\le i\le n-1$, where $\tau_{i-1}>\tau_{i}$,
for which we have $k$ possibilities.

Finally, if  $\tau\in\mathcal{B}_{n-1,k-1,j-1}$
then we put $-n$ either at the end of $\tau$ or between $\tau_{i-1}$
and $\tau_{i}$, $1\le i\le n-1$, where $\tau_{i-1}<\tau_{i}$, for which we have $n-k+1$ possibilities.

Therefore the number of $\sigma\in\mathcal{B}_{n,k,j}$
such that $\Lambda\sigma$ belongs to the set $\mathcal{B}_{n-1,k,j}$,
$\mathcal{B}_{n-1,k-1,j}$, $\mathcal{B}_{n-1,k,j-1}$ or
$\mathcal{B}_{n-1,k-1,j-1}$
is equal to $(k+1)B(n-1, k,j)$, $(n-k)B(n-1, k-1,j)$,
$kB(n-1, k,j-1)$ or $(n-k+1)B(n-1, k-1,j-1)$
respectively. This proves (\ref{bbnkjrecurrence}).

For the boundary conditions it is clear that if $\mathrm{neg}(\sigma)>0$
then $\mathrm{des}(\sigma)>0$, which yields $B(n,0,j)=[j=0]$. We note that the map $(\sigma_0,\sigma_1,\ldots,\sigma_n)\mapsto(\sigma_1,\ldots,\sigma_n)$
is a bijection of $\mathcal{B}_{n,k,0}$ onto $\mathcal{A}_{n,k}$,
consequently $B(n,k,0)=A(n,k)$. For the two others we refer to (\ref{bbnkjsymmetry}).
\end{proof}

Below we present tables for the numbers $B(n,k,j)$ for $n=0,1,2,3,4,5$:
\[
\begin{array}{c|c}
  	k\setminus j & 0  \\\hline
	0   &1
  	  	\end{array},\quad
\begin{array}{c|cc}
  	k\setminus j & 0 &1 \\\hline
	0   &1 & 0  \\
  	1	&0 & 1
  	\end{array},\quad
\begin{array}{c|ccc}
  	k\setminus j & 0 &1 &2 \\\hline
	0   &1 & 0 &0 \\
  	1	&1 & 4 &1 \\
  	2	&0 & 0 &1
  	\end{array},\quad
 	\begin{array}{c|cccc}
 k\setminus j & 0 &1 &2 &3\\\hline
 	0	&1 &0 &0 &0\\
 	1	&4 &12 &6 &1\\
 	2	&1 &6 &12 &4 \\
	3	&0 &0 &0 &1
\end{array},
\]
\[
\begin{array}{c|ccccc}
  		k\setminus j & 0 &1 &2 &3 &4\\\hline	
  		0	&1  &0  &0  &0  &0\\
  		1	&11 &32 &24 &8  &1\\
  		2	&11 &56 &96 &56 &11\\
  		3	&1  &8  &24 &32 &11\\
  		4	&0  &0  &0  &0  &1
  \end{array},\qquad
  \begin{array}{c|cccccc}
  	k\setminus j & 0 &1 &2 &3 &4 &5\\\hline
  	0	&1  &0   &0   &0   &0   &0   \\
  	1	&26 &80  &80  &40  &10  &1   \\
  	2	&66 &330 &600 &480 &180 &26  \\
  	3	&26 &180 &480 &600 &330 &66  \\
  	4	&1  &10  &40  &80  &80  &26  \\
  	5	&0  &0   &0   &0   &0   &1
  	\end{array}.
\]

For example we have $B(n,1,0)=2^{n}-n-1$ and $B(n,1,j)=\binom{n}{j}2^{n-j}$ for $1\le j\le n$
(cf. A038207 in OEIS).
We will see that $B(n,k,j)/\binom{n}{j}$ is always an integer.

\section{Generating functions}\label{sectiongeneratingfunctions}

Now we define three families of polynomials corresponding to the numbers $B(n,k,j)$:
\begin{align}
P_{n,j}(x)&:=\sum_{k=0}^{n} B(n,k,j)x^k,\\
Q_{n,k}(y)&:=\sum_{j=0}^{n} B(n,k,j)y^j,\\
R_{n}(x,y)&:=\sum_{j,k=0}^{n} B(n,k,j)x^k y^j.
\end{align}
The polynomials $R_{n}(x,y)$ were studied by Brenti~\cite{brenti},
who called them ``$q$-Eulerian polynomials of type $B$".

The symmetry (\ref{bbnkjsymmetry}) implies:
\begin{align}
P_{n,j}(x)&=x^n P_{n,n-j}(1/x),\\
Q_{n,k}(y)&=y^n Q_{n,n-k}(1/y),\label{bbqsymmetry}\\
R_{n}(x,y)&=x^n y^n R_{n}(1/x,1/y).
\end{align}

\begin{proposition}\label{propositionprecurrence}
The polynomials $P_{n,j}(x)$ satisfy the following recurrence:
\begin{align}
P_{n,j}(x) &= (1+nx-x)P_{n-1,j}(x) + (x-x^2)P_{n-1,j}'(x)\label{bbpnjrecurrence}\\
&+ nxP_{n-1,j-1}(x)+(x-x^2)P_{n-1,j-1}'(x),\nonumber
\end{align}
with the initial conditions: $P_{n,0}(x)=P^{\mathrm{A}}_{n}(x)$ for $n\ge0$ and
$P_{n,n}(x)=x P^{\mathrm{A}}_{n}(x)$ for $n\ge1$.
\end{proposition}

\begin{proof}
It is easy to verify that
\[
\sum_{k=0}^{n} (k+1) B(n-1,k,j)x^{k}= P_{n-1,j}(x) + xP_{n-1,j}'(x),
\]
\[
\sum_{k=0}^{n} (n-k)B(n-1,k-1,j)x^{k}=nxP_{n-1,j}(x) - xP_{n-1,j}(x) - x^2P_{n-1,j}'(x),
\]
\[
\sum_{k=0}^{n} kB(n-1,k,j-1)x^{k}= xP_{n-1,j-1}'(x),
\]
and
\[
\sum_{k=0}^{n} (n-k+1)B(n-1,k-1,j-1)x^{k}
= nxP_{n-1,j-1}(x) - x^2P_{n-1,j-1}' (x).
\]
Summing up and applying (\ref{bbnkjrecurrence}) we obtain (\ref{bbpnjrecurrence}).
\end{proof}

Br\"{a}nd\'{e}n \cite{branden}, Corollary~6.9, proved that
for every nonempty subset $S\subseteq\{1,\ldots,n\}$
the polynomial $\sum_{j\in S}P_{n,j}(x)$ has only real and simple roots.
Combining (\ref{cpitalplittlep}) with Example~7.8.8 in \cite{branden2}
we will note (Theorem~\ref{theorempnjrealroots}) that in fact every linear combination
$c_{0} P_{n,0}(x)+c_{1} P_{n,1}(x)+\ldots+c_{n} P_{n,n}(x)$, with $c_0,c_1,\ldots,c_n\ge0$, has only real roots.
The cases when $S$ is the set of even or odd numbers in $\{1,\ldots,n\}$ were studied in \cite{borowiec2016}.
The Newton's inequality implies that if $0\le j\le n$ then the sequence $\{B(n,k,j)\}_{k=0}^{n}$
satisfies a stronger version of log-concavity, namely
\begin{equation}\label{bbnkjlogconcavej}
B(n,k,j)^2\ge B(n,k-1,j)B(n,k+1,j)\frac{(k+1)(n-k+1)}{k(n-k)}
\end{equation}
for $0<k<n$, in particular this sequence is unimodal.

For the polynomials $Q_{n,k}(y)$ we have the following, see (18) in \cite{brenti}:

\begin{proposition}
The polynomials $Q_{n,k}(y)$ satisfy the following recurrence:
\[
Q_{n,k}(y)=(k+1+ky)Q_{n-1,k}(y)+(n-k+(n-k+1)y)Q_{n-1,k-1}(y)
\]
with the initial conditions: $Q_{n,0}(y)=1$, $Q_{n,n}(y)=y^n$ for $n\ge0$.
\end{proposition}

The polynomials $Q_{n,k}$ however do not have all roots real.
They satisfy the following versions of Worpitzky identity:
\begin{align}
\sum_{k=0}^{n}\binom{u+n-k}{n}Q_{n,k}(y)&=(u+1+uy)^n,\\
\sum_{k=0}^{n}\binom{u+k}{n}Q_{n,k}(y)&=(u+y+uy)^n.
\end{align}
The former is proved in \cite{brenti}, Theorem~3.4,
the latter follows from the former and the symmetry~(\ref{bbqsymmetry}).

Now we recall the recurrence relation for $R_{n}(x,y)$ (see Theorem~3.4 in \cite{brenti}):

\begin{proposition}
The polynomials $R_{n}(x,y)$ admit the following recurrence:
\[
R_{n}(x,y)=(1+nxy+nx-x)R_{n-1}(x,y)+(x-x^2)(1+y)\frac{\partial}{\partial x}R_{n-1}(x,y),
\]
$n\ge1$, with initial condition  $R_{0}(x,y)=1$.
\end{proposition}

Brenti \cite{brenti} also found the generating function for the numbers $B(n,k,j)$:
\begin{equation}
f(x,y,z):=\sum_{n=0}^{\infty}\frac{R_{n}(x,y)}{n!}z^n
=\frac{(1-x) e^{(1-x)z}}{1-x e^{(1-x)(1+y)z}}.
\end{equation}
Note that
\begin{equation}
f(x,y,z)=f^{\mathrm{A}}\big(x,(1+y)z\big)e^{(x-1)yz}.
\end{equation}

\section{Refined numbers}\label{sectionrefined}

For $0\le k\le n$ and a subset $U\subseteq\{1,2,\ldots,n\}$
we define $\mathcal{B}_{n,k,U}$
as the set of those $\sigma\in\mathcal{B}_{n,k}$ which
have minus sign at $\sigma_i$, $1\le i\le n$, if and only if $|\sigma_i|\in U$.
Therefore we have
\begin{equation}\label{bbbnkusum}
\dot{\bigcup_{\substack{U\subseteq\{1,\ldots,n\}\\|U|=j}}}\mathcal{B}_{n,k,U}=\mathcal{B}_{n,k,j}.
\end{equation}
The cardinality of $\mathcal{B}_{n,k,U}$ will be denoted $b(n,k,U)$.
By convention we put $b(n,-1,U)=b(n,n+1,U):=0$.
It is quite easy to observe boundary conditions.

\begin{proposition}\label{bbbpropositionboundary}
For $n\ge1$, $0\le k\le n$, $U\subseteq\{1,\ldots,n\}$ we have
\begin{align*}
b(n,0,U)&=[U=\emptyset],&b(n,n,U)&=[U=\{1,\ldots,n\}],\\
b(n,k,\emptyset)&=A(n,k),&b(n,k,\{1,\ldots,n\})&=A(n,n-k).
\end{align*}
\end{proposition}

Now we provide a recurrence relation.

\begin{proposition}\label{littlebpropositionrecurrence}
For $0\le k\le n$, $U\subseteq\{1,2,\ldots,n\}$ we have
\begin{equation}\label{bbbnkurecurrence1}
b(n,k,U)=(k+1)\cdot b(n-1,k,U)+(n-k)\cdot b(n-1,k-1,U)
\end{equation}
if $n\notin U$ and
\begin{equation}\label{bbbnkurecurrence2}
b(n,k,U)=k\cdot b(n-1,k,U')+(n-k+1)\cdot b(n-1,k-1,U')
\end{equation}
if $n\in U$, where $U':=U\setminus\{n\}$.
\end{proposition}

\begin{proof}
Both formulas are true when $k=0$ or $k=n$.
Assume that $0<k<n$.
We will apply the same map $\Lambda:\mathcal{B}_{n}\to\mathcal{B}_{n-1}$ as in the proof of Theorem~2.1.
Fix $\sigma\in\mathcal{B}_{n,k,U}$ and assume that $i$ is such that $\sigma_i=n$ (when $n\notin U$)
or $\sigma_i=-n$ (when $n\in U$), $1\le i\le n$. We have now four possibilities:
\begin{itemize}
\item $n\notin U$ and either $i=n$ or $\sigma_{i-1}>\sigma_{i+1}$, $1\le i<n$.
Then $\Lambda\sigma\in\mathcal{B}_{n-1,k,U}$.

\item $n\notin U$ and $\sigma_{i-1}<\sigma_{i+1}$, $1\le i<n$.
Then $\Lambda\sigma\in\mathcal{B}_{n-1,k-1,U}$.

\item $n\in U$ and $\sigma_{i-1}>\sigma_{i+1}$, $1\le i<n$.
Then $\Lambda\sigma\in\mathcal{B}_{n-1,k,U\setminus\{n\}}$.

\item $n\in U$ and either $i=n$ or $\sigma_{i-1}<\sigma_{i+1}$, $1\le i<n$.
Then $\Lambda\sigma\in\mathcal{B}_{n-1,k-1,U\setminus\{n\}}$.
\end{itemize}

On the other hand, as in the proof of Theorem~\ref{theorembrecurrence}, we see that
for a given $\tau$ in $\mathcal{B}_{n-1,k,U}$ (resp. in $\mathcal{B}_{n-1,k-1,U}$)
there are $k+1$ (resp. $n-k$) such $\sigma$'s in $\mathcal{B}_{n,k,U}$ that $\Lambda\sigma=\tau$.
We simply insert $n$ into a descent or at the end of $\tau$
(resp. into an ascent).
Similarly, for a given $\tau$ in $\mathcal{B}_{n-1,k,V}$ (resp. in $\mathcal{B}_{n-1,k-1,V}$)
there are $k$ (resp. $n-k+1$) such $\sigma$'s in $\mathcal{B}_{n,k,V\cup\{n\}}$ that $\Lambda\sigma=\tau$.
\end{proof}

Now we will see that $b(n,k,U)$ depends only on $n,k$ and the cardinality of $U$.

\begin{theorem}\label{theorembnku}
If $0\le k\le n$, $U,V\subseteq\{1,\ldots,n\}$ and $|U|=|V|$ then
\[
b(n,k,U)=b(n,k,V).
\]
\end{theorem}

\begin{proof}
Fix $U,V\subseteq\{1,\ldots,n\}$, with $|U|=|V|$ and define
$\tau\in\mathcal{A}_n$ as the unique permutation of $\{1,\ldots,n\}$
such that: $\tau(U)=V$, $\tau|_{U}$ preserves the order and
$\tau|_{\{1,\ldots,n\}\setminus U}$ preserves the order.
We extend $\tau$ to an element of $\mathcal{B}_{n}$
by putting $\tau(-i)=-\tau(i)$.
Now let $\sigma\in \mathcal{B}_{n,k,U}$.
Then, by definition, $\tau(\sigma(i))<0$ if and only if $\sigma(i)<0$, $-n\le i\le n$.
Moreover, if $1\le i\le n$ then $\tau(\sigma(i-1))<\tau(\sigma(i))$ if and only if $\sigma(i-1)<\sigma(i)$.
This is clear when $\sigma(i-1)$ and $\sigma(i)$ have different signs.
If they have the same sign then this is a consequence of the order preserving property of
$\tau$ on $U$ and on $\{1,\ldots,n\}\setminus U$.
Consequently, the map $\sigma\mapsto\tau\circ\sigma$ is a bijection of $\mathcal{B}_{n,k,U}$
onto $\mathcal{B}_{n,k,V}$.
\end{proof}

The theorem justifies the following definition: for $0\le j,k\le n$ we put
\[
b(n,k,j):=b(n,k,U),
\]
where $U$ is an arbitrary subset of $\{1,\ldots,n\}$ with $|U|=j$.
In addition, if $j<0$ or $k<0$ or $n<j$ or $n<k$ then we put $b(n,k,j)=0$.
From (\ref{bbbnkusum}) we obtain

\begin{corollary}
For $0\le j,k\le n$ we have
\begin{equation}\label{bbbinombb}
\binom{n}{j}b(n,k,j)=B(n,k,j).
\end{equation}
\end{corollary}

\section{Connections with permutations of type $A$}\label{sectiontypea}

For given $n\ge0$ we define a map $F_{n}:\mathcal{A}_{n+1}\to\mathcal{B}_n$
in the following way: $F_{n}(\sigma)=\widetilde{\sigma}$, where
for $1\le i\le n$ we put
\begin{equation}
\widetilde{\sigma}_i:=\left\{\begin{array}{ll}
\sigma_{i+1}-\sigma_{1}&\hbox{if $\sigma_{i+1}<\sigma_{1}$,}\\
\sigma_{i+1}-1&\hbox{if $\sigma_{i+1}>\sigma_1$,}
\end{array}
\right.
\end{equation}
$\widetilde{\sigma}_{-i}:=-\widetilde{\sigma}_{i}$ and $\widetilde{\sigma}_0:=0$.
Note that $\widetilde{\sigma}_{i-1}>\widetilde{\sigma}_{i}$ if and only if $\sigma_{i}>\sigma_{i+1}$ for $1\le i\le n$,
so the number of descents in $(0,\widetilde{\sigma}_1,\ldots,\widetilde{\sigma}_n)$
is the same as in $(\sigma_1,\ldots,\sigma_{n+1})$.
It is easy to see that $F_{n}$ is one-to-one. Its image is the set of such
$\tau\in\mathcal{B}_{n}$ which satisfy the following property:
if $1\le i_1,i_2\le n$, $|\tau_{i_1}|<|\tau_{i_2}|$, $\tau_{i_2}<0$
then $\tau_{i_1}<0$.
Denote
\[
\mathcal{A}_{n,k,j}:=\{\sigma\in\mathcal{A}_{n,k}:\sigma_1=j\}.
\]
The cardinalities of these sets were studied by Conger~\cite{conger2010}, who denoted
$\left\langle n\atop k\right\rangle_j:=\left|\mathcal{A}_{n,k,j}\right|$.

From our remarks we have

\begin{theorem}\label{theorembijection}
For $0\le j,k\le n$ the function $F_n$ maps $\mathcal{A}_{n+1,k}$ into $\mathcal{B}_{n,k}$
and is a bijection from $\mathcal{A}_{n+1,k,j+1}$ onto $\mathcal{B}_{n,k,\{1,\ldots,j\}}$.
Consequently,
\begin{equation}\label{bbbnkjankj}
b(n,k,j)=\left|\mathcal{A}_{n+1,k,j+1}\right|.
\end{equation}
\end{theorem}

In the rest of this section we briefly collect some properties of the numbers $b(n,k,j)=\left\langle n+1\atop k\right\rangle_{j+1}$,
most of them are immediate consequences of the results of Conger \cite{conger2010,congerthesis}.

\begin{proposition}
If $0\le k,j\le n$ then
\begin{align}
b(n,0,j)&=[j=0],\label{bbbcollecta}\\
b(n,n,j)&=[j=n],\label{bbbcollectb}\\
b(n,k,0)&=A(n,k),\label{bbbcollectc}\\
b(n,k,n)&=A(n,n-k),\label{bbbcollectd}\\
b(n,k,j)&=(k+1) b(n-1,k,j)+(n-k) b(n-1,k-1,j),\quad j<n,\label{bbbcollecte}\\
b(n,k,j)&=k b(n-1,k,j-1)+(n-k+1) b(n-1,k-1,j-1),\quad j>0,\label{bbbcollectf}\\
b(n,k,j)&=b(n,n-k,n-j).\label{bblitlesymmetry}
\end{align}
\end{proposition}

\begin{proof}
These formulas are consequences of Proposition~\ref{bbbpropositionboundary},
Proposition~\ref{littlebpropositionrecurrence}, (\ref{bbnkjsymmetry}) and (\ref{bbbinombb})
(see formulas (3) and (8) in \cite{conger2010}).
Note that (\ref{bbbcollectf}) is absent in \cite{conger2010}.
\end{proof}

Applying (\ref{bbbcollecte}), with $j-1$ instead of $j$, and (\ref{bbbcollectf}) we obtain (see (10) in \cite{conger2010})

\begin{corollary} For $1\le j,k\le n$
\begin{equation}
b(n,k,j-1)-b(n,k,j)
=b(n-1,k,j-1)-b(n-1,k-1,j-1).\label{bbbdifference}
\end{equation}
\end{corollary}

Below we present tables for the numbers $b(n,k,j)$ for $n=0,1,2,3,4,5,6$
(they also appear in Appendix~A of \cite{congerthesis}):
\[
\begin{array}{c|c}
  	k\setminus j & 0  \\\hline
	0   &1
  	  	\end{array},\quad
\begin{array}{c|cc}
  	k\setminus j & 0 &1 \\\hline
	0   &1 & 0  \\
  	1	&0 & 1
  	\end{array},\quad
\begin{array}{c|ccc}
  	k\setminus j & 0 &1 &2 \\\hline
	0   &1 & 0 &0 \\
  	1	&1 & 2 &1 \\
  	2	&0 & 0 &1
  	\end{array},\quad
 	\begin{array}{c|cccc}
 k\setminus j & 0 &1 &2 &3\\\hline
 	0	&1 &0 &0 &0\\
 	1	&4 &4 &2 &1\\
 	2	&1 &2 &4 &4 \\
	3	&0 &0 &0 &1
\end{array},
\]
\[
\begin{array}{c|ccccc}
  		k\setminus j & 0 &1 &2 &3 &4\\\hline	
  		0	&1  &0  &0  &0  &0\\
  		1	&11 &8 &4 &2  &1\\
  		2	&11 &14 &16 &14 &11\\
  		3	&1  &2  &4 &8 &11\\
  		4	&0  &0  &0  &0  &1
  \end{array},\qquad
  \begin{array}{c|cccccc}
  	k\setminus j & 0 &1 &2 &3 &4 &5\\\hline
  	0	&1  &0   &0   &0   &0   &0   \\
  	1	&26 &16  &8  &4  &2  &1   \\
  	2	&66 &66 &60 &48 &36 &26  \\
  	3	&26 &36 &48 &60 &66 &66  \\
  	4	&1  &2  &4  &8  &16  &26  \\
  	5	&0  &0   &0   &0   &0   &1
  	\end{array},
\]
\[
\begin{array}{c|ccccccc}
  	k\setminus j & 0 &1 &2 &3 &4 &5&6\\\hline
0& 1&0&0&0&0&0&0\\
1&57&32&16&8&4&2&1\\
2&302&262&212&160&116&82&57\\
3&302&342&372&384&372&342&302\\
4&57&82&116&160&212&262&302\\
5&1&2&4&8&16&32&57\\
6&0&0&0&0&0&0&1
\end{array}.
\]

From (\ref{bbbinombb}), (\ref{bbbcollecte}) and (\ref{bbbcollectf}) we can provide new recurrence formulas for the numbers $B(n,k,j)$:

\begin{corollary}
For $0\le j,k\le n$ we have
\begin{align*}
B(n,k,j)&=\frac{(k+1)n}{n-j}B(n-1,k,j)+\frac{(n-k)n}{n-j}B(n-1,k-1,j),\\
\intertext{if $0\le j<n$ and}
B(n,k,j)&=\frac{k n}{j}B(n-1,k,j-1)+\frac{(n-k+1)n}{j}B(n-1,k-1,j-1),
\end{align*}
{if $0<j\le n$.}
\end{corollary}

Now we introduce the following lexicographic order on the set $\{0,1,\ldots,n\}^2$:
$(k_1,j_1)\preceq (k_2,j_2)$ if and only if either $k_1<k_2$ or $k_1=k_2$, $j_1\ge j_2$.
This is a linear order, in which the successor of $(k,0)$, with $0\le k<n$, is $(k+1,n)$,
and for $1\le j\le n$ the successor of $(k,j)$ is $(k,j-1)$.
It turns out that for every $n\ge1$ the array $\big(b(n,k,j)\big)_{k,j=0}^{n}$
is lexicographically unimodal, cf. Theorem~7 in \cite{conger2010}.

\begin{proposition}\label{propositionbbblexicographic}
For every $n\ge1$ we have the following:

a) If either $0\le k<n/2$, $1\le j\le n$ or $k=n/2$, $n/2<j\le n$ then
\[
b(n,k,j-1)\ge b(n,k,j).
\]
This inequality is sharp unless either $k=0$, $2\le j\le n$ or $n$ is odd, $k=(n-1)/2$, $j=1$.

b) If either $1\le k\le n/2$, $0\le j\le n$ or $n$ is odd, $k=(n+1)/2$, $(n+1)/2\le j\le n$ then
\[
b(n,k-1,j)\le b(n,k,j)
\]
and this inequality is sharp unless $n$ is even, $k=n/2$, $j=0$.

c) The array of numbers $b(n,k,j)$, $0\le j,k\le n$,
is unimodal with respect to the order ``$\preceq$",
with the maximal value $b(n,n/2,n/2)$ if $n$ is even
and \[b(n,(n-1)/2,n)=b(n,(n+1)/2,0)\] if $n$ is odd.
\end{proposition}

\begin{proof}
First we note that (a) implies (c) as a consequence of the symmetry (\ref{bblitlesymmetry})
and the equality
\[
b(n,k-1,0)=A(n,k-1)=A(n,n-k)=b(n,k,n).
\]
Similarly we get (b).

Now assume that the statement holds for $n-1$.
If either $k<n/2$ or $k=n/2$, $n/2<j$ then, due to (3), the right hand side of (\ref{bbbdifference}) is nonnegative
which proves (a), (b) and consequently (c) for $n$.
Moreover, it is positive unless $j=1$, $n-1=2k$, as $A(2k,k-1)=A(2k,k)$.
\end{proof}

Now we note two summation formulas (see (4) and (5) in \cite{conger2010}).

\begin{proposition}
For $0\le j,k\le n$ we have
\begin{align}
\sum_{j=0}^{n}b(n,k,j)&=A(n+1,k),\label{bbbsumj}\\
\sum_{k=0}^{n}b(n,k,j)&=n!.
\end{align}
\end{proposition}

\begin{proof}
For (\ref{bbbsumj}) we apply (\ref{bbbnkjankj}) to the following decomposition:
\[
\mathcal{A}_{n+1,k,1}\dot{\cup}\mathcal{A}_{n+1,k,2}\dot{\cup}\ldots\dot{\cup}\mathcal{A}_{n+1,k,n+1}
=\mathcal{A}_{n+1,k}.
\]
The latter identity is a consequence of (\ref{bbnkjsum2}) and (\ref{bbbinombb}).
\end{proof}

It turns out that (\ref{aformula}) can be generalized to a formula which expresses
the numbers $b(n,k,j)$, see Theorem~1 in \cite{conger2010}.

\begin{theorem}
For any $0\le j,k\le n$ we have
\begin{equation}\label{bbbformula}
b(n,k,j)=\sum_{i=0}^{k}(-1)^{k-i}\binom{n+1}{k-i}i^{j}(i+1)^{n-j},
\end{equation}
under convention that $0^0=1$.
\end{theorem}

\begin{proof}
It can be proved by induction by applying (\ref{aformula}), (\ref{bbbcollectd}) and (\ref{bbbcollectf}).
\end{proof}

From (\ref{bbbformula}) and (\ref{bbbinombb}) we can derive a formula for the numbers $B(n,k,j)$.

\begin{corollary}
For any $0\le j,k\le n$ we have
\begin{equation}
B(n,k,j)=\binom{n}{j}\sum_{i=0}^{k}(-1)^{k-i}\binom{n+1}{k-i}i^{j}(i+1)^{n-j},
\end{equation}
under convention that $0^0=1$.
\end{corollary}

Now we can prove Worpitzky type formula:

\begin{proposition}
For $0\le j\le n$ we have
\begin{equation}\label{bbbworpitzky}
\sum_{k=0}^{n} b(n,k,j)\binom{x+n-k}{n}
=x^{j}(1+x)^{n-j}.
\end{equation}
\end{proposition}

\begin{proof}
If $x\in\{0,1,\ldots,n\}$ then
\[
\sum_{k=0}^{n}(-1)^{k-i}\binom{n+1}{k-i}\binom{x+n-k}{n}=[x=i]
\]
(see (5.25) in \cite{gkp}).
Applying (\ref{bbbformula}) we see that (\ref{bbbworpitzky})
holds for $x\in\{0,1,\ldots,n\}$ (see formula (4.18) in \cite{congerthesis}).
Since the left hand side is a polynomial of degree
at most $n$, this implies that (\ref{bbbworpitzky}) is true for all $x\in\mathbb{R}$.
\end{proof}

\section{Real rootedness}\label{sectionrealroots}

For $0\le j\le n$ denote
\begin{equation}
p_{n,j}(x):=\sum_{k=0}^{n} b(n,k,j)x^k
\end{equation}
so that
\begin{equation}\label{cpitalplittlep}
P_{n,j}(x)=\binom{n}{j}p_{n,j}(x).
\end{equation}
By Proposition~\ref{propositionprecurrence} we have the following recurrence:
\begin{align}
p_{n,j}(x) &=\frac{n-j}{n} (1+xn-x)p_{n-1,j}(x) + \frac{n-j}{n}(x-x^2)p_{n-1,j}'(x)\label{bbpnjrecurrence2}\\
&+ jxp_{n-1,j-1}(x)+\frac{j}{n}(x-x^2)p_{n-1,j-1}'(x),\nonumber
\end{align}
with the initial conditions: $p_{n,0}(x)=P^{\mathrm{A}}_{n}(x)$ for $n\ge0$ and
$p_{n,n}(x)=x P^{\mathrm{A}}_{n}(x)$ for $n\ge1$.
By (\ref{bbbnkjankj}) the polynomial $p_{n,j}(x)$ coincides with $A_{n+1,j+1}(x)$
considered by Br\"{a}nd\'{e}n \cite{branden2}, Example 7.8.8.
He noted that
\begin{equation}\label{brandenrecurrence}
p_{n,j}(x)=\sum_{i=0}^{j-1}xp_{n-1,i}(x)+\sum_{i=j}^{n-1}p_{n-1,i}(x),
\end{equation}
which is equivalent to
\begin{equation}
b(n,k,j)=\sum_{i=0}^{j-1} b(n-1,k-1,i)+\sum_{i=j}^{n-1} b(n-1,k,i)
\end{equation}
(see (9) in \cite{conger2010}).
Note that if $0\le j<n$ then $\deg p_{n,j}(x)=n-1$
and $\deg p_{n,n}(x)=n$. In fact, $p_{n,n}(x)=x p_{n,0}(x)$.
Conger~\cite{conger2010}, Theorem~5, proved that all $p_{n,j}(x)$ have only real roots. It turns out that they
admit a much stronger property.

Let $f,g\in\mathbb{R}[x]$ be real-rooted polynomials with positive leading coefficients.
We say that $f$ is an \textit{interleaver} of $g$, which we denote $f\ll g$, if
\[
\ldots\le\alpha_2\le\beta_2\le\alpha_1\le\beta_1,
\]
where $\{\alpha_i\}_{i=1}^{m}$, $\{\beta_i\}_{i=1}^{n}$
are the roots of $f$ and $g$ respectively.
A sequence $\{f_i\}_{i=0}^{n}$ of real-rooted polynomials is called
\textit{interlacing} if $f_i\ll f_j$ whenever $0\le i<j\le n$.

From \cite{branden2}, Example~7.8.8 and (\ref{cpitalplittlep})
we have the following property
of the polynomials $p_{n,j}(x)$ and $P_{n,j}(x)$:

\begin{theorem}\label{theorempnjrealroots}
For every $n\ge1$
the sequence $\{p_{n,j}(x)\}_{j=0}^{n}$ is interlacing.
Consequently, for any $c_0,c_{1},\ldots,c_n\ge0$ the polynomial
\[
c_{0}p_{n,0}(x)+c_{1}p_{n,1}(x)+\ldots+c_{n}p_{n,n}(x)
\]
has only real roots.

The same statement holds for the polynomials $P_{n,j}(x)$.
\end{theorem}

Note that Theorem~\ref{theorempnjrealroots} generalizes Corollary~3.7 in \cite{brenti} and Corollary~6.9 in \cite{branden}.

\end{document}